\UseRawInputEncoding
\documentclass[12pt]{amsart}
\usepackage{a4wide,enumerate,xcolor,graphicx}
\usepackage{amsmath}
\usepackage{float}
\allowdisplaybreaks

\newcommand{\N}{{\mathbb N}}
\newcommand{\R}{{\mathbb R}}

\makeatletter
\@namedef{subjclassname@2020}{%
	\textup{2020} Mathematics Subject Classification}
\makeatother

\numberwithin{equation}{section}
\newtheorem{theorem}{Theorem}[section]

\newtheorem{proposition}[theorem]{Proposition}

\newtheorem{assumption}{Assumption}

\begin{document}

\title[strong propagation of chaos for multi-species cross-diffusion equations]{Quantitative estimates of propagation of chaos for multi-species cross-diffusion equations }

\author[Y. Li]{Yue Li}
\address{School of mathematics and statistics, Shandong Normal University, Jinan 250358, P. R. China}
\email{liyue2011008@163.com}

\author[Y. Yin]{Yefei Yin}
\address{School of Mathematics and Statistics, The University of Sydney, NSW 2006, Australia}
\email{yyin0132@uni.sydney.edu.au}

\author[Z. Zhang]{Zhipeng Zhang}
\address{School of Mathematical Sciences, Ocean University of China, Qingdao 266100, P. R. China}
\email{zhangzp@ouc.edu.cn}

\date{\today}

\thanks{Y. Li is supported by NSFC (Grant No. 12501268), Taishan Scholars Program of Shandong Province (Grant No. tsqnz20250715) and the Shandong Provincial Natural Science Foundation (Grant No. ZR2025QC1508). Z. Zhang  is supported by NSFC (Grant No. 12471215 and 12331007) and Taishan Scholars Program (Grant No. tsqn202507101)}

\begin{abstract}
In this paper, we prove the quantitative propagation of chaos results that allow us to derive multi-species cross-diffusion equations from moderately interacting stochastic particle system.
The quantitative propagation of chaos result in $L^1$-norm is obtained by the relative entropy method, and the proof is carried out in two steps.
In the first step, we quantify the relative entropy between the joint distribution of the particle system and the tensorised solution of the PDE at the intermediate level.
In the second step, we establish a rigorous convergence rate to the multi-species cross-diffusion equations by analyzing the $L^2$-distance between the solution of the intermediate-level PDE and that of the limiting PDE.
Furthermore, combining the strong $L^1$-convergence for the propagation of chaos with the $L^p$-estimates $(2\le p<\infty)$ for the marginal distribution of multi-species particle system,
we derive the corresponding $L^q$-result $(1<q<\infty)$ via interpolation.
\end{abstract}



\maketitle


\section{Introduction}
In the last decades, an increasing interest in studying the cross-diffusion models, which refer to the systems of quasilinear parabolic equations characterized by a non-diagonal diffusion matrix,
is stimulated by extensive applications in fields such as cell biology, multicomponent gas flow and population dynamics \cite{Ju16}.
In this paper, we focus on the following multi-species cross-diffusion equations
\begin{equation}
\label{macro1}
\left\{
\begin{aligned}
\partial_t u_i-\sigma_i\Delta u_i&={\rm div} \Big(\sum_{j=1}^n a_{ij}u_i\nabla u_j \Big)\quad\mbox{in } \R^d,\;t>0,\\
u_i(0)&=u^0_i, \quad 1\le i\le n,
\end{aligned}
\right.
\end{equation}
where $d\geq 1$ denotes the space dimension,
$i\in\{1,\ldots,n\}$ is the species index, $u^0_i\geq 0$ in $\R^d$, $u=(u_1,\ldots,u_n)$ represents the vector of population densities,
and $\sigma_i>0$ and $a_{ij}\ge 0$ stand for the constant diffusion coefficients and the limit of the interaction potentials respectively.

This paper can be regarded as an extension of \cite{CDJ}, in which the model \eqref{macro1} was rigorously derived from a stochastic many-particle system that governs the motion of $n$ distinct particle species.
For each species $i$, the number of particles is given by $N_i \in \mathbb{N}$ $(1\le i\le n)$. Without loss of generality, we assume $N_i = N$ for all $ i\in\{1,\ldots, n\}$.
Let $(\Omega, \mathcal{F}, (\mathcal{F}_t)_{t \geq 0}, \mathbb{P})$ be a complete filtered probability space.
For each $i\in \{1, \ldots, n\}$ and $k \in\{ 1, \ldots, N\}$, we consider a $d$-dimensional $\mathcal{F}_t$-Brownian motion $(W^k_i(t))_{t \geq 0}$, which are assumed to be independent of each other.
Moreover, let $(\zeta^k_i)$ be independent and identically distributed random variables, independent of $(W^k_i(t))_{t \geq 0}$, with a common probability density function $u_i^0$.
We denote by $X^{k,N}_{\eta,i}(t)$ the position of the $k$-th particle belonging to the $i$-th species at time $t$. Its evolution is governed by the following dynamics
\begin{equation}
\label{micro1}
\left\{
\begin{aligned}
dX_{\eta,i}^{k,N}(t)&=-\sum_{j=1}^n\frac{1}{N}\sum_{l=1}^N\nabla V^{\eta}_{ij}(X^{k,N}_{\eta,i}(t)-X^{l,N}_{\eta,j}(t))dt+\sqrt{2\sigma_i}dW^k_i(t),\\
X^{k,N}_{\eta,i}(0)&=\zeta^k_i, \ \ \ \ 1\le i\le n,\; 1\le k\le N.
\end{aligned}
\right.
\end{equation}
Here the interaction potential $V^\eta_{ij}$ takes the form $V^\eta_{ij}(x)=\frac{1}{\eta^d}V_{ij}(\frac{x}{\eta})$,
where $\eta>0$, $V_{ij}\in C_0^\infty(\R^d)$ satisfying $\mbox{supp}(V_{ij})\subset B_1(0)$ and $\int_{\R^d}V_{ij} dx=a_{ij}$ $(1\le i,j\le n)$.
Therefore, $V_{ij}^\eta\to a_{ij}\delta$ in the sense of distributions as $\eta\to 0$,
with $\delta$ being the Dirac delta distribution.

The systematic derivation of cross-diffusion-type models, however, is a more recent development.
For instance, the hydrodynamic limit of the empirical densities of a two-component Brownian system was shown in \cite{Seo} to be governed by the Maxwell--Stefan equations.
In a related direction, a non-local Lotka--Volterra system featuring cross-diffusion was derived from particle dynamics in \cite{FM2015}, while the Shigesada--Kawasaki--Teramoto system was obtained in \cite{DDD} from a microscopic many-particle Markov process.
Extensions to systems involving more than two species remain scarce. For an arbitrary number of competing populations, the global existence of weak solutions to cross-diffusion systems was established in \cite{CxDJ}.
In the context of moderate interactions with logarithmic scaling, the mean-field limit for multiple population species was obtained in \cite{CDJ}.
Furthermore, the authors of \cite{CG2026} established well-posedness and propagation of chaos for a class of multi-species cross-diffusion equations, and showed their connection to interacting particle systems.
In \cite{CDHJ}, the authors derived cross-diffusion equations of Shigesada--Kawasaki--Teramoto type from stochastic moderately interacting many-particle systems for multiple species.
While the former work \cite{CDJ} established the propagation of chaos only in the weak sense, the goal of our paper is to further prove the propagation of chaos in the strong sense by means of the relative entropy method.

The relative entropy method has been widely employed as a powerful tool for deriving macroscopic equations from interacting particle systems, with a broad range of applications.
In 1991, energy-based estimates were employed in \cite{Oelschlager1991} to rigorously derive the continuity equation and the Euler equation as continuum limits of Hamiltonian many-particle systems.
The relative entropy method was employed in the context of hydrodynamic limits to derive the Stokes equations \cite{Varadhan} and the incompressible Euler equations \cite{S-R}.
Later, the relative entropy was used in \cite{JW2016} to derive the mean-field limit to the Vlasov equations, together with the propagation of chaos, through the strong convergence of all marginals.
Furthermore, \cite{JW2018} proved quantitative propagation of chaos via explicit relative-entropy estimates, requiring only $W^{-1,\infty}$ regularity on the interaction kernel.
By a modulated energy method, the authors of \cite{Serfaty2018,Serfaty2020} obtained the mean-field convergence of point systems with Coulomb-type interactions.
Moreover, \cite{BJW2019} proposed a modulated free energy, combining the methods of \cite{JW2018} and \cite{Serfaty2018,Serfaty2020}, that handles mixed singular potentials and yields quantitative derivations of chemotaxis models such as the Patlak--Keller--Segel system.
In a recent work, the author of \cite{Lacker2023} introduced a BBGKY-based framework  that provides a quantitative propagation of chaos for mean-field diffusive dynamics with optimal $O((k/N)^2)$ relative-entropy bounds on the $k$-particle marginals, improving the previous $O(k/N)$ rate.
The combination of relative entropy and Oelschl\"ager's regularized \(L^2\)-estimate \cite{Oelschlager1987} was employed in \cite{CHH} to derive propagation of chaos for the viscous porous medium equation in the moderate interaction regime.
The authors of \cite{CGH} generalized the relative-entropy convergence theory to the multi-species moderately interacting particle systems up to Newtonian singularity.
Up to our knowledge, quantitative estimates of strong propagation of chaos for multi-species cross-diffusion equations are new.

Using It\^o's inequality, the Kolmogorov forward equation of \eqref{micro1} is given by
\begin{equation}
\label{Kf}
\left\{
\begin{aligned}
\partial_t u_{\eta,n}^N-\sum_{i=1}^n\sum_{k=1}^N\sigma_i\Delta_{x^k_i} u^N_{\eta,n}
&=\sum_{i=1}^n\sum_{k=1}^N{\rm div} _{x^k_i}\Big(\sum_{j=1}^n \frac{1}{N}\sum_{l=1}^N\nabla V^\eta_{ij}(x^k_i-x^l_j)u^N_{\eta,n} \Big),\\
u^N_{\eta,n}(0)&=(u^0_1)^{\otimes N}\otimes \cdots\otimes (u^0_n)^{\otimes N},
\end{aligned}
\right.
\end{equation}
where $u^N_{\eta,n}=u^N_{\eta,n}(t,x_1^1,\ldots,x_1^N,\ldots,x_n^1,\ldots,x_n^N)$ is the joint law of the particles $[(X_{\eta,i}^{k,N})_{i=1}^n]_{k=1}^N$.
For fixed $\eta>0$, the global existence and uniqueness of a classical solution to the linear parabolic problem \eqref{Kf} with smooth coefficients can be obtained as a trivial consequence of classical parabolic theory.

As noted in \cite{CGH}, the definition of the marginal distribution for multi-species particle system requires special care,
since-unlike in the single-species setting-the function $u^N_{\eta,n}$ is not fully symmetric.
Let $r=(r_1,\ldots,r_n)$ with $r_i\in \N$ $(1\le i\le n)$.
For the Liouville equation \eqref{Kf}, the $r$-th marginal distribution across the $n$-species is defined as
\begin{align*}
u^{N,r}_{\eta,n}=\int_{\R^{d\sum_{i=1}^n(N-r_i)}}u^N_{\eta,n} dx_1^{r_1+1}\cdots dx_1^N \cdots dx_n^{r_n+1}\cdots dx_n^N,
\end{align*}
which characterizes the joint probability distribution of the particles
$$X^{1,N}_{\eta,1},\ldots,X^{r_1,N}_{\eta,1};\ldots ;X^{1,N}_{\eta,n},\ldots,X^{r_n,N}_{\eta,n}.$$

As a mean field approximation of \eqref{micro1}, we introduce the intermediate stochastic system, which can be formally viewed as a mean field limit $N\to\infty$ in the system \eqref{micro1} for fixed $\eta>0$,
\begin{equation}
\label{micro2}
\left\{
\begin{aligned}
d\bar{X}_{\eta,i}^{k}(t)&=-\sum_{j=1}^n(\nabla V^\eta_{ij}\ast u_{\eta,j})(\bar{X}^k_{\eta,i}(t),t)dt+\sqrt{2\sigma_i}dW^k_i(t),\\
\bar{X}^{k}_{\eta,i}(0)&=\zeta^k_i, \ \ \ \  1\le i\le n,\; 1\le k\le N,
\end{aligned}
\right.
\end{equation}
where $u_{\eta,j}$ denotes the probability density function of $\bar{X}^k_{\eta,j}$, which is governed by the following non-local differential equations
\begin{equation}
\label{macro2}
\left\{
\begin{aligned}
\partial_t u_{\eta,i}-\sigma_i\Delta u_{\eta,i}&={\rm div} \Big(\sum_{j=1}^n u_{\eta,i}\nabla V^\eta_{ij}\ast u_{\eta,j} \Big)\quad \mbox{in }\R^d,\;t>0,\\
u_{\eta,i}(0)&=u^0_i, \ \ \ \  1\le i\le n.
\end{aligned}
\right.
\end{equation}
Taking the limit $\eta\to 0$, we infer the limiting particle system
\begin{equation}
\label{micro3}
\left\{
\begin{aligned}
d\hat{X}_{i}^{k}(t)&=-\sum_{j=1}^n a_{ij}\nabla u_j (\hat{X}^k_{i}(t),t)dt+\sqrt{2\sigma_i}dW^k_i(t),\\
\hat{X}^{k}_{i}(0)&=\zeta^k_i, \ \ \ \  1\le i\le n,\; 1\le k\le N.
\end{aligned}
\right.
\end{equation}
Here the law of $\hat{X}^k_i$ solves the limiting cross-diffusion equations \eqref{macro1}.

Based on suitable regularity assumptions on the solutions to the PDE equations \eqref{macro1} and \eqref{macro2} and the error estimates between stochastic differential systems \eqref{micro1} and  \eqref{micro3}, the quantitative estimates for propagation of chaos derived in this paper is self-contained.
All these assumptions were rigorously proved in \cite{CDJ} under the parameter condition $\varepsilon\log N\geq \eta^{-2d-4}$,
for sufficiently small $\varepsilon>0$. 

\begin{assumption}\label{ass}
Let $s>d/2+2$ and $\sigma_{min}:=\min_{1\le i\le n}\sigma_i$.
Assume that the problems \eqref{macro1} and \eqref{macro2} possess unique global solutions $u$ and $u_\eta$, respectively, lying in the same function space
$L^\infty(0,T; H^s(\mathbb R^d))\cap L^2(0,T;H^{s+1}(\mathbb R^d))$.
Both $u$ and $u_\eta$ are componentwise nonnegative, satisfy
$\int_{\R^d}u_i(t,x) dx=\int_{\R^d}u_{\eta,i}(t,x) dx=1$ for any $t\in [0,T]$ and $1\le i\le n$,
and obey the common bound
\[
\|v\|_{L^\infty(0,T; H^s(\R^d))}+\|v\|_{L^2(0,T; H^{s+1}(\R^d))}\le \|v(0)\|_{H^s(\R^d)},
\]
with $v=u$ or $v=u_\eta$ {\rm(}where $v(0)=u^0$ in both cases{\rm)}.
Moreover, we require the error estimates between the cross-diffusion equations \eqref{macro1} and the non-local equations \eqref{macro2} as follows
\begin{align}\label{errorpde}
\|u_\eta-u\|_{L^\infty(0,T;L^2(\R^d))}
+\|\nabla (u_\eta-u)\|_{L^2(0,T;L^2(\R^d))}\le C(T,\sigma_{min})\eta.
\end{align}
We also assume the error estimates for the stochastic systems
\begin{align}\label{errorsto2}
\sup_{1\le k\le N}\mathbb{E}\Big[\sum_{i=1}^n \sup_{0<s<T}|X^{k,N}_{\eta,i}(s)-\bar{X}^k_{\eta,i}(s)| \Big]
\le CN^{-1+C\varepsilon},
\end{align}
and
\begin{align}\label{errorsto}
\sup_{1\le k\le N}\mathbb{E}\Big[\sum_{i=1}^n \sup_{0<s<T}|X^{k,N}_{\eta,i}(s)-\hat{X}^k_i(s)| \Big]
\le C\eta,
\end{align}
where the constant $C>0$ in \eqref{errorsto2} and $\eqref{errorsto}$ depends on $T$, $n$, $\|D^2 V_{ij}\|_{L^\infty(\R^d)}$, $\sigma_{min}$ and $u^0$.
\end{assumption}

Under the above assumptions, we state our main results, where we adopt the same logarithmic relation between the parameter $\eta$ and the number of particles $N$ as in \cite{CDJ}.
\begin{theorem}[Quantitative propagation of chaos result in $L^1(\R^{d|r|})$-norm]
	\label{relativeL1}
	Let $|r|:=\sum_{i=1}^n r_i$, and assume that $\int_{\mathbb{R}^d} u^0_i(x)|x|^2\,dx<\infty$ for each $1\le i\le n$.
       Under Assumption \ref{ass}, choose $\varepsilon>0$ sufficiently small {\rm(}depending on $n$, $T$, $\|\nabla V_{ij}\|_{W^{1,\infty}(\R^d)}$, $\sigma_{min}$, $u^0$ and r{\rm)} so that, whenever $\varepsilon\log N \ge \eta^{-2d-4}$, we have
	\begin{align*}
		\sup_{t\in [0,T]}\Big\|u^{N,r}_{\eta,n}-\prod_{i=1}^n u_{i}^{\otimes r_i}\Big\|_{L^1(\R^{d|r|})}\le C\eta^{\frac{4}{4+d}},
	\end{align*}
where $C>0$ depends on $n$, $T$, $\|\nabla V_{ij}\|_{W^{1,\infty}(\R^d)}$, $\sigma_{min}$, $u^0$ and $r$.
\end{theorem}

Based on Theorem \ref{relativeL1} and the $L^p$-estimates for $u^{N,r}_{\eta,n}$, we can further obtain the propagation of chaos result in $L^q(\R^{d|r|})$-norm.
\begin{theorem}[Quantitative propagation of chaos result in $L^q(\R^{d|r|})$-norm]\label{relativeLq}
Let $|r|:=\sum_{i=1}^n r_i$. Assume that, for each $1\le i\le n$, the initial data satisfy $u_i^0\in L^p(\mathbb{R}^d)$ with $2\le p<\infty$ and $\int_{\mathbb{R}^d} u^0_i(x)|x|^2\,dx<\infty$.
       Under Assumption \ref{ass}, choose $\varepsilon>0$ sufficiently small {\rm(}depending on $n$, $T$, $\|\nabla V_{ij}\|_{W^{1,\infty}(\R^d)}$, $\sigma_{min}$, $u^0$ and r{\rm)} so that, whenever $\varepsilon\log N \ge \eta^{-2d-4}$, it holds
\begin{align*}
\sup_{t\in [0,T]}\Big\|u^{N,r}_{\eta,n}-\prod_{i=1}^n u_i^{\otimes r_i}\Big\|_{L^q(\R^{d|r|})}\le C\eta^{\frac{2}{q}} \quad \mbox{for any } q\ge 2,
\end{align*}
and
\begin{align*}
\sup_{t\in [0,T]}\Big\|u^{N,r}_{\eta,n}-\prod_{i=1}^n u^{\otimes r_i}_i\Big\|_{L^q(\R^{d|r|})}\le C\eta^{\frac{2(q-1)}{q}}\quad \mbox{for any } 1<q<2,
\end{align*}
where $C>0$ depends on $n$, $T$, $\|\nabla V_{ij}\|_{W^{1,\infty}(\R^d)}$, $\sigma_{min}$, $u^0$ and $r$.
\end{theorem}

The proof is divided into two steps.
 In the first step, we deduce a mean-field type estimate between the joint distribution of the particle system \eqref{micro1} and
the tensorised solution $u_\eta^{\otimes Nn}:=\prod_{i=1}^n u_{\eta,i}^{\otimes N}$ of the intermediate PDE equations \eqref{macro2}.
A simple computation gives
\begin{equation}
\label{relativeueta}
\left\{
\begin{aligned}
\partial_t u_\eta^{\otimes Nn}-\sum_{i=1}^n \sum_{k=1}^N\sigma_i\Delta_{x^k_i}u_\eta^{\otimes Nn}
&=\sum_{i=1}^n\sum_{k=1}^N{\rm div}_{x^k_i}\Big(\sum_{j=1}^n\nabla V^\eta_{ij}\ast u_{\eta,j}(x_i^k)u_\eta^{\otimes Nn} \Big),  \\
u_\eta^{\otimes Nn}(0)&=(u_1^0)^{\otimes N}\otimes \cdots \otimes (u_n^0)^{\otimes N}.
\end{aligned}
\right.
\end{equation}
The main tool used in this step is the relative entropy, whose definition is given as follows.
For $\ell\in\N$ and two probability density functions $f$ and $g$ on $\R^{d\ell}$, let
\begin{align*}
\mathcal{H}(f|g)
:=\int_{\R^{d\ell}}f\log\frac{f}{g}\,dx_1\cdots dx_\ell.
\end{align*}
The $L^1$-distance can be controlled by the relative entropy
\begin{align}\label{formula}
\frac{1}{2}\Big\|u^{N,r}_{\eta,n}-\prod_{i=1}^n u_{\eta,i}^{\otimes r_i}\Big\|_{L^1(\R^{d|r|})}^2
\le \mathcal{H}\Big(u^{N,r}_{\eta,n}\big|\prod_{i=1}^n u_{\eta,i}^{\otimes r_i} \Big)
\le \frac{\max_{1\le i\le n} r_i}{N}\mathcal{H}\Big(u^{N}_{\eta,n}\big| u_{\eta}^{\otimes Nn} \Big).
\end{align}
The proof of \eqref{formula} follows from the Csisz\'ar-Kullback-Pinsker inequality and the definition of marginal distribution (cf. \cite[Lemma 2.2]{CGH} and \cite[Lemma 3.9]{MM}); we therefore omit the details for brevity.
Actually, we show in the second step, a mean-field type estimate between the tensorised solution of the PDE at the intermediate level and that one at the limiting level
is obtained using the error estimates \eqref{errorpde}.
The proof of Theorem \ref{relativeL1} is finished by combining the first and second steps.

Using the quantitative propagation of chaos in $L^1$-norm together with the $L^p$-estimates for $u^{N,r}_{\eta,n}$,
we further conclude the corresponding result in the $L^q$-norm by interpolation.

This paper is organized as follows. Section 2 is devoted to the proof of Theorem \ref{relativeL1}.
In Section 3, we complete the proof of Theorem \ref{relativeLq}.

\section{Proof of Theorem \ref{relativeL1}}
To prove Theorem \ref{relativeL1}, we start with the estimate between the $r$-th marginal distribution $u^{N,r}_{\eta,n}$ of the particle system \eqref{micro1} and
the tensorised solution $\prod_{i=1}^n u_{\eta,i}^{\otimes r_i}$ of the intermediate PDE equations \eqref{macro2}, which is presented as follows.
\begin{proposition}\label{relative1}
Under Assumption \ref{ass}, choose $\varepsilon>0$ sufficiently small {\rm(}depending on $n$, $T$, $\|\nabla V_{ij}\|_{W^{1,\infty}(\R^d)}$, $\sigma_{min}$, $u^0$ and r{\rm)} so that, whenever $\varepsilon\log N \ge \eta^{-2d-4}$, it holds
\begin{align*}
\sup_{t\in [0,T]}\Big\|u^{N,r}_{\eta,n}-\prod_{i=1}^n u_{\eta,i}^{\otimes r_i}\Big\|_{L^1(\R^{d|r|})}^2\le CN^{-1+C\varepsilon},
\end{align*}
where $C$ is a positive constant depends on $n$, $T$, $\|\nabla V_{ij}\|_{W^{1,\infty}(\R^d)}$, $\sigma_{min}$, $u^0$ and r.
\end{proposition}
\begin{proof}
We deduce from \eqref{Kf} and \eqref{relativeueta} that
\begin{align*}
&\frac{d}{dt}\frac{1}{N}\mathcal{H}(u^N_{\eta,n}|u_{\eta}^{\otimes Nn})\nonumber\\
=&\frac{1}{N}\int_{\R^{dNn}}\sum_{i=1}^n\sum_{k=1}^N\Big(-\sigma_i\nabla _{x^k_i}u^N_{\eta,n}-\frac{1}{N}\sum_{j=1}^n\sum_{l=1}^N\nabla V^\eta_{ij}(x_i^k-x_j^l)u^N_{\eta,n} \Big)\cdot
\nabla _{x^k_i}\log \frac{u^N_{\eta,n}}{u_\eta^{\otimes Nn}} dx_1^1\cdots dx_n^N\nonumber\\
&-\frac{1}{N}\int_{\R^{dNn}}\sum_{i=1}^n\sum_{k=1}^N\Big(-\sigma_i\nabla_{x^k_i}u_\eta^{\otimes Nn}-\sum_{j=1}^n\nabla V^\eta_{ij}\ast u_{\eta,j}(x^k_i)u_\eta^{\otimes Nn} \Big)
\cdot \nabla _{x^k_i}\frac{u^N_{\eta,n}}{u_\eta^{\otimes Nn}} dx_1^1\cdots dx_n^N\nonumber\\
=&\frac{1}{N}\int_{\R^{dNn}}\sum_{i=1}^n\sum_{k=1}^N\sigma_i \Big(-\nabla _{x^k_i}u^N_{\eta,n}+\frac{u^N_{\eta,n}}{u_\eta^{\otimes Nn}}\nabla_{x^k_i}u_\eta^{\otimes Nn} \Big)
\cdot \nabla_{x^k_i}\log \frac{u^N_{\eta,n}}{u_\eta^{\otimes Nn}}dx_1^1\cdots dx_n^N\nonumber\\
&-\frac{1}{N}\int_{\R^{dNn}}\sum_{i=1}^n\sum_{k=1}^N \sum_{j=1}^n\Big(\frac{1}{N}\sum_{l=1}^N \nabla V^{\eta}_{ij}(x^k_i-x^l_j)
-\nabla V^\eta_{ij}\ast u_{\eta,j}(x_i^k) \Big)u^N_{\eta,n}\cdot \nabla_{x^k_i}\log \frac{u^N_{\eta,n}}{u_\eta^{\otimes Nn}} dx_1^1\cdots dx_n^N\\
=&-\frac{1}{N}\int_{\R^{dNn}}\sum_{i=1}^n\sum_{k=1}^N\sigma_i \Big|\nabla_{x^k_i}\log \frac{u^N_{\eta,n}}{u_\eta^{\otimes Nn}} \Big|^2 u^N_{\eta,n}
dx_1^1\cdots dx_n^N\nonumber\\
&-\frac{1}{N}\int_{\R^{dNn}}\sum_{i=1}^n\sum_{k=1}^N \sum_{j=1}^n\Big(\frac{1}{N}\sum_{l=1}^N \nabla V^{\eta}_{ij}(x^k_i-x^l_j)
-\nabla V^\eta_{ij}\ast u_{\eta,j}(x_i^k) \Big)u^N_{\eta,n}\cdot \nabla_{x^k_i}\log \frac{u^N_{\eta,n}}{u_\eta^{\otimes Nn}} dx_1^1\cdots dx_n^N.
\end{align*}
By Young's inequality, the last term on the right-hand side of the above equality can be divided into several parts.
\begin{align*}
&\frac{d}{dt}\frac{1}{N}\mathcal{H}(u^N_{\eta,n}|u_{\eta}^{\otimes Nn})\nonumber\\
\leq& -\frac{1}{2N}\int_{\R^{dNn}}\sum_{i=1}^n\sum_{k=1}^N \sigma_i \Big|\nabla_{x^k_i}\log \frac{u^N_{\eta,n}}{u_\eta^{\otimes Nn}} \Big|^2u^N_{\eta,n} dx_1^1\cdots dx_n^N\nonumber\\
&+C\mathbb{E}\Big[\frac{1}{N}\sum_{i=1}^n\sum_{k=1}^N\Big|\frac{1}{N}\sum_{j=1}^n\sum_{l=1}^N \nabla V^\eta_{ij}(X^{k,N}_{\eta,i}-X^{l,N}_{\eta,j})
-\frac{1}{N}\sum_{j=1}^n\sum_{l=1}^N\nabla V^\eta_{ij}(\bar{X}^k_{\eta,i}-\bar{X}^l_{\eta,j})\Big|^2 \Big]\nonumber\\
&+C\mathbb{E}\Big[\frac{1}{N}\sum_{i=1}^n\sum_{k=1}^N\Big|\frac{1}{N}\sum_{j=1}^n\sum_{l=1}^N\nabla V^\eta_{ij}(\bar{X}^k_{\eta,i}-\bar{X}^l_{\eta,j})
-\sum_{j=1}^n\nabla V^\eta_{ij}\ast u_{\eta,j}(\bar{X}^k_{\eta,i})\Big|^2 \Big]\nonumber\\
&+C\mathbb{E}\Big[\frac{1}{N}\sum_{i=1}^n\sum_{k=1}^N\Big|\sum_{j=1}^n\nabla V^\eta_{ij}\ast u_{\eta,j}(\bar{X}^k_{\eta,i})
-\sum_{j=1}^n\nabla V^\eta_{ij}\ast u_{\eta,j}(X^{k,N}_{\eta,i})\Big|^2 \Big]\nonumber\\
=:& -\frac{1}{2N}\int_{\R^{dNn}}\sum_{i=1}^n\sum_{k=1}^N \sigma_i \Big|\nabla_{x^k_i}\log \frac{u^N_{\eta,n}}{u_\eta^{\otimes Nn}} \Big|^2 u^N_{\eta,n} dx_1^1\cdots dx_n^N
+I_1+I_2+I_3.
\end{align*}
Using mean value theorem, H\"older's inequality and Jensen's inequality, the term $I_1$ can be estimated as follows.
\begin{align*}
I_1\le& \frac{C}{N}\sup_{1\le i,j\le n}\|D^2V^\eta_{ij}\|^2_{L^\infty(\R^d)}\sum_{k=1}^N
\mathbb{E}\Big[\sum_{i=1}^n\Big||X^{k,N}_{\eta,i}-\bar{X}^k_{\eta,i}|+\frac{1}{N}\sum_{j=1}^n\sum_{l=1}^N|X^{l,N}_{\eta,j}-\bar{X}^l_{\eta,j}| \Big|^2 \Big]\nonumber\\
\le& C\eta^{-2(d+2)}\|D^2V_{ij}\|^2_{L^\infty(\R^d)}\Big(\frac{1}{N}\sum_{k=1}^N\mathbb{E}\Big[\sum_{i=1}^n|X^{k,N}_{\eta,i}-\bar{X}^k_{\eta,i}|^2 \Big]
+\mathbb{E}\Big[\Big|\frac{1}{N}\sum_{j=1}^n\sum_{l=1}^N |X^{l,N}_{\eta,j}-\bar{X}^l_{\eta,j}| \Big|^2 \Big]\Big)\nonumber\\
\le& \frac{C}{N}\eta^{-2(d+2)}\sum_{k=1}^N\mathbb{E}\Big[\sum_{i=1}^n |X_{\eta,i}^{k,N}-\bar{X}^k_{\eta,i}|^2 \Big]
\le C\eta^{-2(d+2)}\sup_{1\le k\le N}\mathbb{E}\Big[\sum_{i=1}^n |X_{\eta,i}^{k,N}-\bar{X}^k_{\eta,i}|^2 \Big],
\end{align*}
together with \eqref{errorsto2}, we infer that
\begin{align}\label{I1}
I_1\le C\eta^{-2(d+2)}N^{-1+C\varepsilon}.
\end{align}
It follows from Young's convolution inequality and the $L^\infty(0,T;L^1(\R^d))$ bound for $u_{\eta,j}$ $(1\le j\le n)$ that
\begin{align*}
\|D^2V^{\eta}_{ij}\ast u_{\eta,j}\|_{L^\infty((0,T)\times\R^d)}
\le C\|D^2 V^{\eta}_{ij}\|_{L^\infty(\R^d)}\|u_{\eta,j}\|_{L^\infty(0,T;L^1(\R^d))}\le C\eta^{-(d+2)},
\end{align*}
which implies that
\begin{align}
I_3\le& \frac{C}{N}\sup_{1\le i,j\le n}\|D^2V^{\eta}_{ij}\ast u_{\eta,j}\|^2_{L^\infty((0,T)\times\R^d)}
\sum_{k=1}^N\mathbb{E}\Big[\sum_{i=1}^n |X^{k,N}_{\eta,i}-\bar{X}^k_{\eta,i}|^2 \Big]\nonumber\\
\le& C\eta^{-2(d+2)}\sup_{1\le k\le N}\mathbb{E}\Big[\sum_{i=1}^n|X^{k,N}_{\eta,i}-\bar{X}^k_{\eta,i}|^2  \Big]
\le C\eta^{-2(d+2)}N^{-1+C\varepsilon}.
\end{align}
To deal with the term $I_2$, we introduce the following notation.
\begin{align*}
Z^{k,l}_{i,j}(s):=\nabla V^\eta_{ij}(\bar{X}^k_{\eta,i}-\bar{X}^l_{\eta,j})-\nabla V^\eta_{ij}\ast u_{\eta,j}(\bar{X}^k_{\eta,i}).
\end{align*}
Then, it holds
\begin{align*}
\Big|\frac{1}{N}\sum_{j=1}^n\sum_{l=1}^N Z^{k,l}_{i,j}(s) \Big|^2
\le \frac{C}{N^2}\sum_{j=1}^n \Big(\sum_{l,m=1}^N Z^{k,l}_{i,j}(s)Z^{k,m}_{i,j}(s) \Big).
\end{align*}
We start with the case $l\neq m$ and $i\neq j$.
\begin{align*}
\mathbb{E}[Z^{k,l}_{i,j}(s)Z^{k,m}_{i,j}(s)]=
&\int_{\R^d}\int_{\R^d}\int_{\R^d} \nabla V^\eta_{ij}(z-x)\nabla V^\eta_{ij}(z-y)u_{\eta,i}(z)u_{\eta,j}(x) u_{\eta,j}(y) dxdz dy\\
&+\int_{\R^d}\Big(\int_{\R^d}\nabla V^\eta_{ij}(z-x)u_{\eta,j}(x) dx \Big)\Big(\int_{\R^d}\nabla V^\eta_{ij}(z-y)u_{\eta,j}(y) dy \Big)u_{\eta,i}(z)dz\\
&-2\int_{\R^d}\int_{\R^d}\nabla V^\eta_{ij}(z-x)\nabla V^\eta_{ij}\ast u_{\eta,j}(z)u_{\eta,i}(z)u_{\eta,j}(x)dxdz=0.
\end{align*}
A similar argument can be applied to deduce that, for the case $l\neq m$, $k\neq l$, $k\neq m$ and $i=j$,
$\mathbb{E}[Z^{k,l}_{i,i}(s)Z^{k,m}_{i,i}(s)]=0$ as well.
Now we turn to the case $l\neq m$, $i=j$ and $k=l$.
\begin{align*}
\mathbb{E}[Z^{k,k}_{i,i}(s)Z^{k,m}_{i,i}(s)]=
&\mathbb{E}\big[\big(\nabla V^\eta_{ii}(0)-\nabla V^\eta_{ii}\ast u_{\eta,i}(\bar{X}^k_{\eta,i}(s),s)\big)\\
&\cdot\big(\nabla V^\eta_{ii}(\bar{X}^k_{\eta,i}(s)-\bar{X}^m_{\eta,i}(s))-\nabla V^\eta_{ii}\ast u_{\eta,i}(\bar{X}^k_{\eta,i}(s),s)  \big) \big]=0.
\end{align*}
In a similar way, we show that $\mathbb{E}[Z^{k,l}_{i,i}(s)Z^{k,k}_{i,i}(s)]=0$, for $l\neq m$, $i=j$ and $k=m$.
Therefore, we have
\begin{align*}
\Big|\frac{1}{N}\sum_{j=1}^n\sum_{l=1}^N Z^{k,l}_{i,j}(s) \Big|^2
\le \frac{C}{N^2}\sum_{j=1}^n\sum_{l=1}^N |Z^{k,l}_{i,j}(s)|^2,
\end{align*}
this shows that
\begin{align}\label{I3}
I_2=& \frac{C}{N}\sum_{i=1}^n\sum_{k=1}^N\mathbb{E}\Big[\Big|\frac{1}{N}\sum_{j=1}^n\sum_{l=1}^N Z^{k,l}_{i,j}(s) \Big|^2 \Big]
\le \frac{C}{N}\sum_{i=1}^n\sum_{k=1}^N \mathbb{E}\Big[\frac{1}{N^2}\sum_{j=1}^n\sum_{l=1}^N |Z^{k,l}_{i,j}(s)|^2 \Big]\nonumber\\
\le& \frac{C}{N}\sup_{1\le i,j\le n}(\|\nabla V^{\eta}_{ij}\|^2_{L^\infty(\R^d)}+\|\nabla V^{\eta}_{ij}\ast u_{\eta,j}\|^2_{L^\infty((0,T)\times\R^d)})\nonumber\\
\le& \frac{C}{N}\eta^{-2(d+1)},
\end{align}
where we have used $\|\nabla V^{\eta}_{ij}\|^2_{L^\infty(\R^d)}+\|\nabla V^{\eta}_{ij}\ast u_{\eta,j}\|^2_{L^\infty((0,T)\times\R^d)}\le C\eta^{-2(d+1)}$ in the last inequality.
We combine \eqref{I1}--\eqref{I3} to derive that
\begin{align*}
&\frac{d}{dt}\frac{1}{N}\mathcal{H}(u^N_{\eta,n}|u_{\eta}^{\otimes Nn})
+\sum_{i=1}^n\sum_{k=1}^N \frac{\sigma_i}{2N}\int_{\R^{dNn}}\Big|\nabla_{x^k_i}\log \frac{u^N_{\eta,n}}{u_\eta^{\otimes Nn}} \Big|^2u^N_{\eta,n} dx_1^1\cdots dx^N_n\\
\le& C\eta^{-2(d+2)}N^{-1+C\varepsilon}+C\eta^{-2(d+1)}N^{-1}\le CN^{-1+C\varepsilon},
\end{align*}
where we have used the assumption $\varepsilon\log N\geq \eta^{-2(d+2)}$.
This allows us to use \eqref{formula} to finish the proof of Proposition \ref{relative1}.
\end{proof}

\vskip5mm
Now we turn to the $L^1$-error estimate for $\prod_{i=1}^n u_{\eta,i}^{\otimes r_i}-\prod_{i=1}^n u_{i}^{\otimes r_i}$,
where $\prod_{i=1}^n u_{\eta,i}^{\otimes r_i}$ and $\prod_{i=1}^n u_{i}^{\otimes r_i}$ are the tensorised solutions of the intermediate PDE equations \eqref{macro2}
and the limiting PDE equations \eqref{macro1} respectively.
\begin{proposition}\label{relative2}
Let $\int_{\mathbb{R}^d} u^0_i(x)|x|^2\,dx<\infty$ for each $1\le i\le n$.
       Under Assumption \ref{ass}, we have
\begin{align}
\sup_{t\in [0,T]}\Big\|\prod_{i=1}^n u_{\eta,i}^{\otimes r_i}-\prod_{i=1}^n u_{i}^{\otimes r_i}\Big\|_{L^1(\R^{d|r|})}\le C\eta^{\frac{4}{4+d}},
\end{align}
where $C>0$ depends on $T$, $n$, $u^0$ and $r$.
\end{proposition}
\begin{proof}
A multiplication of \eqref{macro2} by $|x|^2$ and integration over $\R^d$ lead to
\begin{align*}
\frac{d}{dt}\int_{\R^d} u_{\eta,i}|x|^2 dx
=&\sigma_i \int_{\R^d}u_{\eta,i}\Delta|x|^2 dx-\int_{\R^d}\sum_{j=1}^n u_{\eta,i}\nabla V^\eta_{ij}\ast u_{\eta,j}\cdot \nabla |x|^2dx\\
=& 2d\sigma_i\int_{\R^d}u_{\eta,i} dx-2\sum_{j=1}^n \int_{\R^d}u_{\eta,i}\nabla V^\eta_{\ij}\ast u_{\eta,j}\cdot x dx.
\end{align*}
Young's convolution inequality and the $L^\infty(0,T;H^s(\R^d))$ $(s>d/2+2)$ bound for $u_{\eta,j}$ give
\begin{align*}
\|\nabla V^\eta_{ij}\ast u_{\eta,j}\|_{L^\infty((0,T)\times\R^d)}\le C\|V^\eta_{ij}\|_{L^1(\R^d)}\|\nabla u_{\eta,j}\|_{L^\infty((0,T)\times\R^d)}\le C.
\end{align*}
Thus, in accordance with the $L^\infty(0,T;L^1(\R^d))$ bound for $u_{\eta,i}$ and Young's inequality, we achieve that
\begin{align*}
&\frac{d}{dt}\int_{\R^d} u_{\eta,i}|x|^2 dx\\
\le& 2d\sigma_i\int_{\R^d}u_{\eta,i} dx
+C\sum_{j=1}^n\|\nabla V^\eta_{\ij}\ast u_{\eta,j}\|_{L^\infty((0,T)\times\R^d)}\Big(\int_{\R^d}u_{\eta,i} dx+\int_{\R^d}u_{\eta,i}|x|^2 dx \Big)\\
\le& C+C\int_{\R^d}u_{\eta,i}|x|^2 dx,
\end{align*}
this allows us to use Gronwall's inequality to get
\begin{align*}
\sup_{t\in [0,T]}\int_{\R^d}u_{\eta,i}|x|^2 dx\le C.
\end{align*}
In a similar way, we show that the $L^\infty(0,T;L^1(\R^d))$-norm of $u_i |x|^2$ is bounded as well.
Together with \eqref{errorpde}, we have
\begin{align*}
\int_{\R^d}|u_{\eta,i}-u_i| dx
\le& \int_{B_R(0)}|u_{\eta,i}-u_i| dx+\int_{B_R(0)^c}|u_{\eta,i}-u_i|\frac{|x|^2}{R^2} dx\\
\le& CR^{\frac{d}{2}}\|u_{\eta,i}-u_i\|_{L^2(\R^d)}+\frac{1}{R^2}\int_{B_R(0)^c}|u_{\eta,i}-u_i||x|^2 dx\\
\le& CR^{\frac{d}{2}}\eta+CR^{-2}.
\end{align*}
Choosing $R^{-1}=\eta^{\frac{2}{4+d}}$, we infer that
\begin{align}\label{relative3}
\sup_{t\in [0,T]}\int_{\R^d}|u_{\eta,i}-u_i| dx\le C\eta^{\frac{4}{4+d}},
\end{align}
finishing the proof of Proposition \ref{relative2}.
\end{proof}
We combine Proposition \ref{relative1} and Proposition \ref{relative2} to conclude that
\begin{align*}
\sup_{t\in [0,T]}\Big\|u^{N,r}_{\eta,n}-\prod_{i=1}^n u_{i}^{\otimes r_i}\Big\|_{L^1(\R^{d|r|})}
\le CN^{-\frac{1}{2}+\hat{C}\varepsilon}+C\eta^{\frac{4}{4+d}},
\end{align*}
where $\hat C>0$ depends on $n$, $T$, $\|\nabla V_{ij}\|_{W^{1,\infty}(\R^d)}$, $\sigma_{min}$, $u^0$ and r.
As long as $\varepsilon < \frac{1}{2\hat{C}}$, the condition $\varepsilon\log N \ge \eta^{-2(d+2)}$ implies
\[
(-\tfrac{1}{2}+\hat{C}\varepsilon)\log N
\le (-\tfrac{1}{2}\varepsilon^{-1}+\hat{C})\eta^{-2(d+2)}.
\]
Consequently,
\[
N^{-\frac{1}{2}+\hat{C}\varepsilon}
\le \exp\left\{ (-\tfrac{1}{2}\varepsilon^{-1}+\hat{C})\eta^{-2(d+2)} \right\}
\le \eta^{\frac{4}{4+d}},
\]
where the final bound follows from the parameter condition $\varepsilon < \frac{1}{2\hat{C}}$.
Therefore, it holds
\begin{align*}
\sup_{t\in [0,T]}\Big\|u^{N,r}_{\eta,n}-\prod_{i=1}^n u_{i}^{\otimes r_i}\Big\|_{L^1(\R^{d|r|})}
\le C\eta^{\frac{4}{4+d}}.
\end{align*}

\section{Proof of Theorem \ref{relativeLq}}
To complete the proof of Theorem \ref{relativeLq}, we start with the $L^p$-estimates $(2\le p<\infty)$ for $u^{N,r}_{\eta,n}$.
It follows from the equations \eqref{Kf} that $u^{N,r}_{\eta,n}$ satisfies
\begin{equation}
\label{macro5}
\left\{
\begin{aligned}
\partial_t u^{N,r}_{\eta,n}=&\sum_{i=1}^n  \sum_{k=1}^{r_i}\sigma_i \Delta_{x^k_i}u^{N,r}_{\eta,n}\\
&+\sum_{i=1}^n\sum_{k=1}^{r_i} {\rm{div}}_{x^k_i}\Big(\int_{\R^{d(Nn-|r|)}}\sum_{j=1}^n \frac{1}{N}\sum_{l=1}^N \nabla V^\eta_{ij}(x^k_i-x^l_j)u^N_{\eta,n}
\prod_{i=1}^n\prod_{k=r_i+1}^N dx^k_i \Big),\\
u^{N,r}_{\eta,n}(0)=&(u_1^0)^{\otimes r_1}\otimes\cdots\otimes (u_n^0)^{\otimes r_n}.
\end{aligned}
\right.
\end{equation}
Multiplying \eqref{macro5} by $p(u^{N,r}_{\eta,n})^{p-1}$ with $p\ge 2$ and integrating over $\R^{d|r|}$, we obtain that
\begin{align*}
&\frac{d}{dt}\int_{\R^{d|r|}}|u^{N,r}_{\eta,n}|^p \prod_{i=1}^n\prod_{k=1}^{r_i}dx^k_i\\
&+\sum_{i=1}^n \sigma_ip(p-1)\sum_{k=1}^{r_i}\int_{\R^{d|r|}}|u^{N,r}_{\eta,n}|^{p-2}|\nabla _{x^k_i}u^{N,r}_{\eta,n}|^2 \prod_{i=1}^n\prod_{k=1}^{r_i}dx^k_i\\
=& -p\sum_{i=1}^n\sum_{k=1}^{r_i}
\int_{\R^{d|r|}}\Big(\int_{\R^{d(Nn-|r|)}}\sum_{j=1}^n\frac{1}{N}\sum_{l=1}^N \nabla V^\eta_{ij}(x^k_i-x^l_j)u^N_{\eta,n}\prod_{i=1}^n\prod_{k=r_i+1}^N dx^k_i  \Big)\\
&\qquad\qquad\qquad\cdot \nabla _{x^k_i}(u^{N,r}_{\eta,n})^{p-1} \prod_{i=1}^n\prod_{k=1}^{r_i}dx^k_i.
\end{align*}
By H\"older's inequality and Young's inequality, we derive that
\begin{align*}
&\frac{d}{dt}\int_{\R^{d|r|}}|u^{N,r}_{\eta,n}|^p \prod_{i=1}^n\prod_{k=1}^{r_i}dx^k_i\\
&+\sum_{i=1}^n \sigma_ip(p-1)\sum_{k=1}^{r_i}\int_{\R^{d|r|}}|u^{N,r}_{\eta,n}|^{p-2}|\nabla _{x^k_i}u^{N,r}_{\eta,n}|^2 \prod_{i=1}^n\prod_{k=1}^{r_i}dx^k_i\\
\le& \sum_{i=1}^n \frac{\sigma_ip(p-1)}{2}\sum_{k=1}^{r_i}\int_{\R^{d|r|}}|u^{N,r}_{\eta,n}|^{p-2}|\nabla _{x^k_i}u^{N,r}_{\eta,n}|^2
\prod_{i=1}^n\prod_{k=1}^{r_i}dx^k_i\\
&+C\sup_{1\le i,j\le n}\|\nabla V^\eta_{ij}\|^2_{L^\infty(\R^d)}
\int_{\R^{d|r|}}|u^{N,r}_{\eta,n}|^{p-2}
\Big(\int_{\R^{d(Nn-|r|)}}u^N_{\eta,n} \prod_{i=1}^n\prod_{k=r_i+1}^N dx^k_i \Big)^2 \prod_{i=1}^n\prod_{k=1}^{r_i}dx^k_i\\
\le& \sum_{i=1}^n \frac{\sigma_ip(p-1)}{2}\sum_{k=1}^{r_i}
\int_{\R^{d|r|}}|u^{N,r}_{\eta,n}|^{p-2}|\nabla _{x^k_i}u^{N,r}_{\eta,n}|^2  \prod_{i=1}^n\prod_{k=1}^{r_i}dx^k_i\\
&+C\eta^{-2(d+1)}
\int_{\R^{d|r|}}|u^{N,r}_{\eta,n}|^{p}  \prod_{i=1}^n\prod_{k=1}^{r_i}dx^k_i,
\end{align*}
where $C>0$ appeared in this section depends on $n$, $T$, $\|\nabla V_{ij}\|_{W^{1,\infty}(\R^d)}$, $\sigma_{min}$, $u^0$ and r.
In view of $u^0_i\in L^p(\R^d)$ $(1\le i\le n)$ and Gronwall's inequality, we have
\begin{align*}
\sup_{t\in[0,T]}\int_{\R^{d|r|}}|u^{N,r}_{\eta,n}|^p  \prod_{i=1}^n\prod_{k=1}^{r_i}dx^k_i
\le C\exp\{CT\eta^{-2(d+1)}\}\le C\exp\{CT\varepsilon\log N\}\le CN^{CT\varepsilon}.
\end{align*}
Due to the interpolation inequality and Proposition \ref{relative1}, we deduce that
\begin{align*}
\Big\|u^{N,r}_{\eta,n}-\prod_{i=1}^n u_{\eta,i}^{\otimes r_i}\Big\|_{L^q(\R^{d|r|})}
\le& C\Big\|u^{N,r}_{\eta,n}-\prod_{i=1}^n u_{\eta,i}^{\otimes r_i}\Big\|_{L^{2q}(\R^{d|r|})}^{\frac{2(q-1)}{2q-1}}
\Big\|u^{N,r}_{\eta,n}-\prod_{i=1}^n u_{\eta,i}^{\otimes r_i}\Big\|_{L^1(\R^{d|r|})}^{\frac{1}{2q-1}}\nonumber\\
\le& CN^{\frac{1}{2q-1}(CT\varepsilon\frac{q-1}{q}+C\varepsilon-\frac{1}{2})}\le CN^{\frac{1}{2q-1}(-\frac{1}{2}+\tilde{C}\varepsilon)},
\end{align*}
where $\tilde C>0$ depends on $n$, $T$, $\|\nabla V_{ij}\|_{W^{1,\infty}(\R^d)}$, $\sigma_{min}$, $u^0$ and r.
As long as $\varepsilon<\frac{1}{2\tilde{C}}$, we can use the assumption $\varepsilon\log N\geq \eta^{-2(d+2)}$ to conclude that
\begin{align}\label{relative4}
\Big\|u^{N,r}_{\eta,n}-\prod_{i=1}^n u_{\eta,i}^{\otimes r_i}\Big\|_{L^q(\R^{d|r|})}
\le C\eta.
\end{align}

Taking into account the error estimates \eqref{errorpde}, the $L^\infty(0,T;H^s(\R^d))$ $(s>d/2+2)$ bounds for $u_{\eta,i}$ and $u_i$ and the interpolation inequality, we infer that, for any $2\le q<\infty$,
\begin{align*}
\|u_{\eta,i}-u_i\|_{L^\infty(0,T;L^q(\R^d))}
\le C\|u_{\eta,i}-u_i\|^{\frac{2}{q}}_{L^\infty(0,T;L^2(\R^d))}\|u_{\eta,i}-u_i\|^{1-\frac{2}{q}}_{L^\infty((0,T)\times\R^d)}\le C\eta^{\frac{2}{q}},
\end{align*}
which shows that
\begin{align*}
\Big\|\prod_{i=1}^n u_{\eta,i}^{\otimes r_i}-\prod_{i=1}^n u_i^{\otimes r_i}\Big\|_{L^\infty(0,T;L^q(\R^{d|r|}))}\le C\eta^{\frac{2}{q}}.
\end{align*}
For the case $1<q<2$, in view of the error estimates \eqref{errorpde}, the $L^\infty(0,T;L^1(\R^d))$ bounds for $u_{\eta,i}$ and $u_i$ and the interpolation inequality, we have
\begin{align*}
\|u_{\eta,i}-u_i\|_{L^\infty(0,T;L^q(\R^d))}\le C\|u_{\eta,i}-u_i\|^{\frac{2-q}{q}}_{L^\infty(0,T;L^1(\R^d))}\|u_{\eta,i}-u_i\|^{\frac{2(q-1)}{q}}_{L^\infty(0,T;L^2(\R^d))}
\le C\eta^{\frac{2(q-1)}{q}},
\end{align*}
therefore,
\begin{align*}
\Big\|\prod_{i=1}^n u_{\eta,i}^{\otimes r_i}-\prod_{i=1}^n u^{\otimes r_i}_i\Big\|_{L^\infty(0,T;L^q(\R^{d|r|}))}\le C\eta^{\frac{2(q-1)}{q}}.
\end{align*}
In accordance with the estimate \eqref{relative4}, we finish the proof of Theorem \ref{relativeLq}.

\vspace{2mm}
\textbf{Conflict of interest.} The authors do not have any possible conflicts of interest.

\vspace{2mm}

\textbf{Data availability statement.}
 Data sharing is not applicable to this article as no data sets were generated or analyzed during the current study.


\begin{thebibliography}{99}

\bibitem{BJW2019} D. Bresch, P.-E. Jabin and Z. Wang, On mean-field limits and quantitative estimates with a large class of
      singular kernels: application to the Patlak--Keller--Segel model, C. R. Math., 357 (2019), 708--720.

\bibitem{CG2026} J.-A. Carrillo and S. Guo, Interacting particle approximation of cross-diffusion systems, Nonlinearity, 39 (2026), 025009.

\bibitem{CGH} J.-A. Carrillo, S. Guo and A. Holzinger, Propagation of chaos for multi-species moderately interacting particle systems up to Newtonian singularity,
       	arXiv:2501.03087.

\bibitem{CDHJ} L. Chen, E.S. Daus, A. Holzinger and A. J\"ungel, Rigorous derivation of population cross-diffusion systems from moderately interacting particle systems,
                 J. Nonlinear Sci., 31 (2021), Paper No. 94, 38 pp.

\bibitem{CDJ} L. Chen, E.S. Daus and A. J\"ungel, Rigorous mean-field limit and cross-diffusion, Z. Angew. Math. Phys., 70 (2019), Paper No. 122, 21pp.

\bibitem{CHH} L. Chen, A. Holzinger and X. Huo, Quantitative convergence in relative entropy for a moderately interacting particle system on $\R^d$, arXiv:2311.01980.

\bibitem{CxDJ} X. Chen, E.S. Daus and A. J\"ungel, Global existence analysis of cross-diffusion population systems for multiple species, Arch. Rational Mech. Anal., 227 (2018), 715--747.

\bibitem{DDD} E.S. Daus, L. Desvillettes and H. Dietert, About the entropic structure of detailed balanced multi-species cross-diffusion equations, J. Differ. Equ., 266 (2019), 3861--3882.

\bibitem{FM2015} J. Fontbona and S. M\'el\'eard, Non local Lotka--Volterra system with cross-diffusion in an heterogeneous medium, J. Math. Biol., 70 (2015), 829--854.

\bibitem{JW2016} P.-E. Jabin and Z. Wang, Mean field limit and propagation of chaos for Vlasov systems with bounded forces, J. Funct. Anal., 271 (2016), 3588--3627.

\bibitem{JW2018} P.-E. Jabin and Z. Wang, Quantitative estimates of propagation of chaos for stochastic systems with $W^{-1,\infty}$ kernels, Invent. Math., 214 (2018), 523--591.

\bibitem{Ju16} A. J\"ungel, Entropy methods for diffusive partial differential equations, BCAM SpringerBriefs, Springer, 2016.

\bibitem{Lacker2023} D. Lacker, Hierarchies, entropy, and quantitative propagation of chaos for mean field diffusions, Probab. Math. Phys., 4 (2023), 377--432.

\bibitem{MM} L. Miclo and P. Del Moral, Genealogies and increasing propagation of chaos for feynman-kac and genetic models, Ann. Appl. Probab., 11 (2001), 1166--1198.

\bibitem{Oelschlager1987} K. Oelschl\"ager, A fluctuation theorem for moderately interacting diffusion processes, Probab. Theory Relat. Fields, 74 (1987), 591--616.

\bibitem{Oelschlager1991} K. Oelschl\"ager, On the connection between Hamiltonian many-particle systems and the hydrodynamical equations, Arch. Ration. Mech. Anal., 115 (1991), 297--310.

\bibitem{S-R} L. Saint-Raymond, Hydrodynamic limits of the Boltzmann equation, Lecture Notes in Mathematics 1971, Springer, 2009.

\bibitem{Seo} I. Seo, Scaling limit of two-component interacting Brownian motions, Ann. Prob., 46 (2018), 2038--2063.

\bibitem{Serfaty2018} S. Serfaty, Systems of points with coulomb interactions. Proceedings of the International Congress of Mathematicians (Rio de Janeiro, 2018), World Scientific, Singapore, 2018, 935--977.

\bibitem{Serfaty2020} S. Serfaty, Mean field limit for coulomb-type flows, Duke Math. J., 169 (2020), 2887--2935.

\bibitem{Varadhan} S. Varadhan, Relative entropy and hydrodynamic limits, Stochastic Processes, Springer, New York, 1993, 329--336.

\end{thebibliography}
\end{document}